\documentclass[10pt,bezier]{article}
\usepackage{amsmath,amssymb,amsfonts,graphicx,caption,subcaption}
\usepackage[colorlinks,linkcolor=red,citecolor=blue]{hyperref}

\newtheorem{prethm}{{\bf Theorem}}[section]
\newenvironment{thm}{\begin{prethm}{\hspace{-0.5
em}{\bf.}}}{\end{prethm}}
\newtheorem{prepro}{{\bf Theorem}}

\newtheorem{precor}[prethm]{{\bf Corollary}}
\newenvironment{cor}{\begin{precor}{\hspace{-0.5
em}{\bf.}}}{\end{precor}}
\newtheorem{preconj}[prethm]{{\bf Conjecture}}

\newtheorem{preremark}[prethm]{{\bf Remark}}
\newenvironment{remark}{\begin{preremark}\em{\hspace{-0.5
em}{\bf.}}}{\end{preremark}}
\newtheorem{prelem}[prethm]{{\bf Lemma}}
\newenvironment{lem}{\begin{prelem}{\hspace{-0.5
em}{\bf.}}}{\end{prelem}}
\newtheorem{preque}[prethm]{{\bf Question}}

\newtheorem{preobserv}[prethm]{{\bf Observation}}

\newtheorem{predef}[prethm]{{\bf Definition}}

\newtheorem{preproposition}[prethm]{{\bf Proposition}}

\newtheorem{preproof}{{\bf Proof.}}
\newtheorem{preprooff}{{\bf Proof}}

\newenvironment{proof}[1]{\begin{preproof}{\rm
#1}\hfill{$\Box$}}{\end{preproof}}

\newtheorem{preproofs}{{\bf The second proof of }}

\newtheorem{preprooft}{{\bf Third proof of }}

\newtheorem{preproofF}{{\bf Proof of}}

\title{\bf\Large 
Highly tree-connected complementary modulo factors with bounded degrees
}
\author{{\normalsize{\sc Morteza Hasanvand${}$} }\vspace{3mm}
\\{\footnotesize{${}$\it Department of Mathematical
 Sciences, Sharif
University of Technology, Tehran, Iran}}
{\footnotesize{}}\\{\footnotesize{ $\mathsf{morteza.hasanvand@alum.sharif.edu }$ }}}

\date{}

\begin{document}
\maketitle
\begin{abstract}{
Let $G$ be a bipartite graph with bipartition $(X,Y)$, let $k$ be a positive integer, and let $f:V(G)\rightarrow Z_k$ be a mapping with $\sum_{v\in X}f(v) \stackrel{k}{\equiv}\sum_{v\in Y}f(v)$. In this paper, we show that if $G$ is $(2m+2m_0+4k-4)$-edge-connected and $m+m_0>0$, then $G$ has an $m$-tree-connected factor $H$ such that its complement is $m_0$-tree-connected and for each vertex $v$, $d_H(v)\stackrel{k}{\equiv} f(v)$, and $$\lfloor\frac{d_G(v)}{2}\rfloor-(k-1)-m_0\le d_{H}(v)\le \lceil\frac{d_G(v)}{2}\rceil+k-1+m.$$ Next, we generalize this result to general graphs and derive a sufficient degree condition for a highly edge-connected general graph $G$ to have a connected factor $H$ such that for each vertex $v$, $d_H(v)\in \{f(v),f(v)+k\}$. Finally, we show that every $(4k-2)$-tree-connected graph admits a bipartite connected  factor whose degrees are divisible by $k$.
\\
\\
\noindent {\small {\it Keywords}: Modulo factor;  spanning Eulerian; partition-connected; connected factor; vertex degree; regular graph. }} {\small
}
\end{abstract}
%
%
%
%
%
%
%
%
%
%
\section{Introduction}
In this article, graphs may have loops and multiple edges.
Let $G$ be a graph. The vertex set, the edge set, and the maximum degree of vertices of $G$ are denoted by $V(G)$, $E(G)$, and $\Delta(G)$, respectively.
We denote by $d_G(v)$ the degree of a vertex $v$ in the graph $G$.
If $G$ has an orientation, the out-degree and in-degree of $v$ are denoted by $d_G^+(v)$ and $d_G^-(v)$.
For a vertex set $A$ of $G$ with at least two vertices, the number of edges of $G$ with exactly one end in $A$ is denoted by $d_G(A)$.
Also, we denote by $e_G(A)$ the number of edges with both ends in $A$.
We denote by $G[A]$ the induced subgraph of $G$ with the vertex set $A$ containing
precisely those edges of $G$ whose ends lie in $A$, and denote by $G[A, B]$ the induced bipartite factor of $G$
with the bipartition $(A, B)$.
A graph $G$ is called 
{\it $m$-tree-connected}, if it contains $m$ edge-disjoint spanning trees. 
Note that by the result of Nash-Williams~\cite{Nash-Williams-1961} and Tutte~\cite{Tutte-1961} every $2m$-edge-connected graph is $m$-tree-connected. 
A graph $G$ is said to be 
{\it $(m,l_0)$-partition-connected},
 if it can be decomposed into an $m$-tree-connected factor and a factor $F$ which admits an orientation such that for each vertex $v$, $d^+_F(v)\ge l_0(v)$, where $l_0$ is an integer-valued function on $V(G)$. 
The {\it bipartite index $bi(G)$} of a graph $G$ is the smallest number of all $|E(G)\setminus E(H)|$ taken over all bipartite factors~$H$. 
An {\it $f$-factor} refers to a spanning subgraph $H$ such that for each vertex $v$, 
$d_H(v)\stackrel{k}{\equiv}f(v) $, where $f:V(G)\rightarrow Z_k$ and $Z_k$ denotes the cyclic group of order $k$.
For a graph $G$, we say that a mapping $f:V(G)\rightarrow Z_k$ is 
{\it compatible with $G$} with respect to a bipartition $X,Y$ of $V(G)$, 
if $\sum_{v\in X} f(v)-2x\stackrel{k}{\equiv} \sum_{v\in Y}f(v)$ for an integer $x$ with $0\le x\le e_G(X)$ 
or $\sum_{v\in X} f(v)\stackrel{k}{\equiv} \sum_{v\in Y}f(v)-2y$ for an integer $y$ with $0\le y\le e_G(Y)$. 
Likewise, we say that a mapping $f$ is 
{\it compatible with $G$}, if it is compatible with $G$ with respect to every bipartition $X,Y$ of $V(G)$.
When $G$ is bipartite and $(2k-1)$-edge-connected, $f$ is compatible with $G$ if and only if $f$ is compatible with $G$
with respect to the unique bipartition $(X,Y)$ of $G$~\cite{ModuloFactorBounded}.
Also, it is easy to see that $f$ is compatible with $G$, if $G$ has an $f$-factor or $bi(G)\ge k-1$ and $(k-1)\sum_{v\in V(G)}f(v)$ is even, see~\cite{ModuloFactorBounded}.
Note that when $f$ is compatible with $G$ and $k$ is even, we must have $\sum_{v\in V(G)} f(v)\stackrel{2}{\equiv} 0$.
A {\it modulo $k$-regular} graph refers to a graph whose degrees are positive and divisible by $k$. 
For a graph $G$ with a vertex $z$, we denote by $\chi_z$ the mapping $\chi_z:V(G)\rightarrow \{0,1\}$
such that $\chi_z(z)=1$ and $\chi_z(v)=0$ for all vertices $v$ with $v\neq z$.
Moreover, we denote by $\bar{\chi}_z$ the mapping $1-\chi_z$.
Throughout this article, all variables $m$ are nonnegative integers and all variables $k$ are positive integers.
%

%

In 1979 Jaeger~\cite{Jaeger-1979} established the first result on connected partity factors by proving that 
 every $4$-edge-connected graph admits a spanning Eulerian subgraph. 
In this paper, we improve this result to the following version by imposing a bound on vertex degrees.
\begin{thm}\label{Intro:thm:Eulerian}
{Every $4$-edge-connected graph $G$ has a spanning Eulerian subgraph $H$ such that for each vertex~$v$,
$$\lfloor \frac{d_{G}(v)}{2}\rfloor-1\le d_H(v)\le \lceil \frac{d_{G}(v)}{2}\rceil+2.$$
}\end{thm}

In 2014 Thomassen~\cite{Thomassen-2014} introduced modulo factors and formulated the following theorem.
Recently, the present author~\cite{ModuloFactorBounded} established a strengthened version for his result by giving a sharp bound on degrees.
%
\begin{thm}{\rm(\cite{ModuloFactorBounded,Thomassen-2014})}\label{Inro:factor:Zk}
{Let $G$ be a bipartite graph, let $k$ be a positive integer, and let $f:V(G)\rightarrow Z_k$ be a compatible mapping. 
If $G$ is $(3k-3)$-edge-connected, then it admits an $f$-factor
$H$ such that for each vertex~$v$,
$$\lfloor \frac{d_G(v)}{2} \rfloor -(k-1) \le d_H(v) \le \lceil \frac{d_G(v)}{2}\rceil+(k-1).$$
}\end{thm}

In this paper, we develop Theorem~\ref{Inro:factor:Zk} to the following 
 tree-connected version.
\begin{thm}\label{Intro:thm:bi}
{Let $G$ be a bipartite graph, let $k$ be a positive integer, and let $f:V(G)\rightarrow Z_k$ be a compatible mapping. 
If $G$ is $(2m+4k-4)$-edge-connected, then it has
an $m$-tree-connected $f$-factor $H$ such that for each vertex $v$,
$$\lfloor \frac{d_G(v)}{2} \rfloor -(k-1) \le d_H(v) \le \lceil \frac{d_G(v)}{2}\rceil+(k-1)+m.$$
}\end{thm}

In Section~\ref{sec:general-graphs}, we will also generalize Theorem~\ref{Intro:thm:bi} to general graphs by proving the following theorem; the special case $m=0$ of this result was investigated in~\cite{ModuloFactorBounded,Thomassen-Wu-Zhang-2016} for even and odd $k$. Moreover, we develop both of Theorems~\ref{Intro:thm:bi} and~\ref{Intro:thm:non-bi} to a decomposition version as mentioned in the abstract.
\begin{thm}\label{Intro:thm:non-bi}
{Let $G$ be a graph, let $k$ be a positive integer, and let $f:V(G)\rightarrow Z_k$ be a compatible mapping. 
If $G$ is $(2m+6k-5)$-tree-connected and $m>0$, then it has
an $m$-tree-connected $f$-factor $H$ such that for each vertex~$v$,
$$\lfloor \frac{d_G(v)}{2} \rfloor -(k-1) \le d_H(v) \le \lceil \frac{d_G(v)}{2}\rceil+(k-1)+m.$$
}\end{thm}

In~\cite{Thomassen-2014} Thomassen used Theorem~\ref{Inro:factor:Zk} to concluded that every $(12k-7)$-edge-connected graph of even order has a bipartite modulo $k$-regular factor whose degrees are not divisible by $2k$. This result is improved in \cite{ModuloFactorBounded} for the existence of bipartite modulo $k$-regular factors as the following theorem. 
In Section~\ref{sec:modulo-regular}, we show that every $(4k-2)$-tree-connected graph admits a bipartite connected modulo $k$-regular factor.
\begin{thm}{\rm(\cite{ModuloFactorBounded})}
{Every $(6k-7)$-edge-connected graph admits a bipartite modulo $k$-regular factor.
}\end{thm}

Recently, the present author investigated tree-connected factors with small degrees in edge-connected graphs and concluded the following theorem. 
\begin{thm}{\rm(\cite{ClosedTrails})}\label{thm:base}
{Every $2m$-edge-connected graph $G$ has an $m$-tree-connected factor $H$
 including an arbitrary given factor with maximum degree at most $m$
 such that for each vertex $v$,
$$d_H(v) \le \lceil \frac{d_G(v)}{2}\rceil+m.$$
}\end{thm}

A weaker version of Theorem~\ref{thm:base} is developed in \cite{AHO} in two ways as the next theorems with the same needed edge-connectivity.
In Sections~\ref{sec:lists} and~\ref{sec:bounded}, we improve them by imposing $H$ to have an arbitrary given factor with maximum degree at most $m$. Moreover, we develop them to a decomposition version.
\begin{thm}{\rm(\cite{AHO})}
{Every $2m$-edge-connected graph $G$ with $m> 0$ has an $m$-tree-connected factor $H$ such that for each vertex $v$,
$$\lfloor \frac{d_G(v)}{2}\rfloor \le d_H(v) \le \lceil \frac{d_G(v)}{2}\rceil+m.$$
}\end{thm}
\begin{thm}{\rm(\cite{AHO})}
{Let $G$ be a $2m$-edge-connected graph.
For each vertex $v$, let $L(v)\subseteq \{m,\ldots, d_G(v)\}$.
Then $G$ has an $m$-tree-connected $L$-factor, 
if for each vertex $v$, 
$$
|L(v)|\ge \lceil\frac{d_G(v)}{2}\rceil+1.$$
}\end{thm}
%
%
%
%
%
%
%
%
\section{Basic tools}
In this section, we shall state some basic tools for working with edge-connected graphs.
For our purpose, the following theorem would be sufficient but for making minor refinements, we need to improve this result slightly.
\begin{thm}{\rm (\cite{ClosedTrails})}\label{thm:basic:M}
{Let $G$ be a graph with a loopless factor $M$ satisfying $\Delta(M)\le m$.
If $G$ is $2m$-edge-connected, then it has an $m$-tree-connected factor $H$ including $M$ such that $G\setminus E(H)$ is $(0,l)$-partition-connected, where for each vertex $v$, $$l(v)= \lfloor d_{G}(v)/2\rfloor-m.$$
Furthermore, for an arbitrary given vertex $z$, we can have $l(z)=\lceil d_{G}(z)/2\rceil$.
}\end{thm}
Before stating the main result, let us recall the following lemmas from~\cite{ClosedTrails,West-Wu-2012}.
\begin{lem}{\rm (\cite{ClosedTrails})}\label{lem:M}
{Let $G$ be a graph with a factor $M$ satisfying $\Delta(M)\le m$ and let $l$ be a nonnegative integer-valued function on $V(G)$.
If $G$ is $(m,l)$-partition-connected, then it has an $m$-tree-connected factor $H$ containing $M$ such that the complement of it is $(0,l)$-partition-connected.
}\end{lem}
\begin{lem}{\rm (\cite{West-Wu-2012})}\label{lem:criterion}
{Let $G$ be a graph and let $l$ be a nonnegative integer-valued function on $V(G)$.
Then $G$ is $(m,l)$-partition-connected if and only if $e_{G}(P)\ge m(|P|-1)+\sum_{\{v\}\in P}l(v)$, where $e_G(P)$
denotes the number of edges of $G$ joining different parts of $P$ plus the number of loops with both ends $v$ with $\{v\}\in P$.
}\end{lem}
The following theorem provides a minor improvement for Theorem~\ref{thm:basic:M}.
\begin{thm}\label{thm:basic:MM0}
{Let $G$ be a graph with two edge-disjoint loopless factors $M$ and $M_0$ satisfying $\Delta(M)\le m$ and $|E(M_0)|\le m$. 
If $G$ is $2m$-edge-connected, then it has an $m$-tree-connected factor $H$ including $M$ excluding $M_0$ such that $G\setminus E(H\cup M_0)$ is $(0,l)$-partition-connected, where for each vertex $v$, $$l(v)= \lfloor d_{G}(v)/2\rfloor-m.$$
Furthermore, for an arbitrary given vertex $z$, we can have $l(z)=\lceil d_{G}(z)/2\rceil -m$.
}\end{thm}
\begin{proof}
{Let $G_0=G\setminus E(M_0)$.
 Let $P$ be a partition of $V(G_0)$. 
Since $G$ is $2m$-edge-connected, we must have
 $$e_{G_0}(P)= e_{G}(P)-e_{M_0}(P)\ge
 \sum_{X\in P}\frac{1}{2}d_{G}(X)-m\ge
 m|P|-m+\sum_{\{v\}\in P}(\frac{1}{2}d_{G}(v)-m).$$
 This implies that $e_{G_0}(P)\ge m(|P|-1)+\sum_{\{v\}\in P}l(v)$, where $l(z)=\lceil d_G(z)/2\rceil-m$ and 
$l(v)=\lfloor d_G(v)/2\rfloor-m$ for each vertex $v$ with $v\neq z$.
By Lemma~\ref{lem:criterion}, the graph $G$ must be $(m,l)$-partition-connected. 
Thus by Lemma~\ref{lem:criterion}, the graph $G_0$ has an $m$-tree-connected factor $H$ containing $M$ such that the complement of it is $(0,l)$-partition-connected.
}\end{proof}
\begin{cor}\label{cor:orientation}
{Let $G$ be a graph with two edge-disjoint loopless factors $M$ and $M_0$ satisfying $\Delta(M)\le m$ and $|E(M_0)|\le m$. 
 For each vertex $v$, let $r(v)$ be a nonnegative integer satisfying $\sum_{v\in V(G)}r(v)=m-|E(M_0)|$.
If $G$ is $2m$-edge-connected, then it has an $m$-tree-connected factor $F$ including $M$ excluding $M_0$ 
and a given pre-orientation of $M_0$ can be extended to an orientation of $G$ 
 such that for each vertex $v$, $d^+_{G}(v)\le \lceil d_{G}(v)/2\rceil$, and 
$$d^-_{F}(v)= m-r(v)-d^-_{M_0}(v).$$
Furthermore, for an arbitrary given vertex $z$, we can have $d^+_{G}(z)\le \lfloor d_{G}(z)/2\rfloor$.
}\end{cor}
\begin{proof}
{Since $G$ is $2m$-edge-connected, one can select a loopless factor $M_0'$ of $G\setminus E(M \cup M_0)$ with a given orientation such that for each vertex $v$, $d^-_{M_0'}(v)=r(v)$.
By Theorem~\ref{thm:basic:MM0}, the graph $G$ 
has an $m$-tree-connected factor $F$ including $M$ excluding $M_0\cup M'_0$ such that $G\setminus E(F\cup M_0\cup M'_0)$ is $(0,l)$-partition-connected, where for each vertex $v$, $l(v)= \lfloor d_{G}(v)/2\rfloor-m$.
Furthermore, we can have $l(z)=\lceil d_{G}(z)/2\rceil -m$.
We may assume that $F$ is minimally $m$-tree-connected and so it is the union of $m$ edge-disjoint spanning trees.
Thus it is not hard to check that there is an orientation for $F$ such that for each vertex $v$,
$d^-_{F}(v)= m-r(v)-d^-_{M_0}(v)$, since $\sum_{v\in V(G)}(r(v)+d^-_{M_0}(v))=m$.
Since $G\setminus E(F\cup M_0 \cup M'_0)$ is $(0,l)$-partition-connected, it has an orientation such that the in-degree of each vertex $v$ is at least $l(v)$. Now, it is enough to induce these orientations to the edges of $G$ to complete the proof.
Note that for each vertex $v$, $d^-_G(v)\ge l(v)+d^-_F(v)+d^-_{M_0\cup M'_0}(v)\ge \lfloor d_{G}(v)/2\rfloor$, and also $d^-_G(z)\ge \lceil d_{G}(z)/2\rceil$.
}\end{proof}
\section{Factors with given dense lists on degrees}
\label{sec:lists}
\subsection{Preliminary results}
In this section, we consider the existence of connected factors with given dense lists on degrees in graphs.
For this purpose, we need the following result on the existence of list factors in directed graphs.
\begin{thm}{\rm (\cite{Frank-Lau-Szabo-2008, Shirazi-Verstraete-2008})}\label{thm:list:base-version}
{Let $G$ be a directed graph and for each vertex $v$, let
$ L(v)\subseteq \{0,\ldots, d_G(v)\} $.
Then $G$ has an $L$-factor, if for each vertex $v$,
$|L(v)| \ge d_G^+(v)+1$.
}\end{thm}
Before stating the main result, let us derive the following generalization of Theorem~\ref{thm:list:base-version}. 
\begin{cor}\label{cor:list:F-F0} 
{Let $G$ be a directed graph with two edge-disjoint factors $F$ and $F_0$.
For each vertex~$v$, take $s(v)$ and $s_0(v)$ to be two integers with $s(v)\le d_{F}(v)$ and $s_0(v)\le d_{F_0}(v)$, and let
$ L(v)\subseteq \{s(v),\ldots, d_G(v)-s_0(v)\} $. 
Then $G$ has an $L$-factor $H$ including $F$ excluding $F_0$, if for each vertex~$v$,
$$|L(v)| \ge d_{G}^+(v)+1+d^-_{F}(v)+d^-_{F_0}(v)-s(v)-s_0(v).$$
}\end{cor}
%
\begin{proof}
{Put $G'=G\setminus E(F\cup F_0)$.
For each vertex $v$, define 
 $L'(v)=\{x-d_{F}(v):\; x\in L(v) \text{ and }d_{F}(v) \le x\le d_G(v)-d_{F_0}(v) \}$ so that 
$L'(v)\subseteq \{0,\ldots, d_{G'}(v)\}.$ 
According this definition, 
$|L'(v)| \ge |L(v)|-(d_{F}(v)-s(v))-(d_{F_0}(v)-s_0(v))$ which implies that 
 $|L'(v)| \ge d^+_{G}(v)+1- d^+_{F_0}(v)- d^+_{F}(v) = d^+_{G'}(v)+1.$
Therefore, by Theorem~\ref{thm:list:base-version},
the graph $G'$ has a factor $H'$ such that for each vertex $v$, $d_{H'}(v)\in L'(v)$. 
Put $H=
H'\cup F$.
For each vertex $v$, we must have $d_{H}(v)=d_{H'}(v)+ d_{F}(v)\in L(v)$.
Hence the graph $H$ is the desired factor.
}\end{proof}
%
\subsection{Graphs with tree-connectivity at least $m+m_0$}
The following theorem gives a sufficient partition-connectivity condition for the existence of tree-connected factors 
with given dense lists on degrees.
\begin{thm}\label{thm:list:partition}
{Let $G$ be an $(m+m_0, l_0)$-partition-connected graph where $l_0$ is a nonnegative integer-valued function on $V(G)$.
For each vertex $v$, let $L(v)\subseteq \{m,\ldots, d_G(v)-m_0\}$.
Then $G$ has an $m$-tree-connected $L$-factor $H$ such that its complement is $m_0$-tree-connected, 
if for each vertex $v$, 
$$|L(v)|\ge d_G(v)+1-l_0(v)-m-m_0.$$
When $m_0=0$, $H$ can include an arbitrary given factor with maximum degree at most $m$.
}\end{thm}
\begin{proof}
{Decomposed $G$ into 
an $m_0$-tree-connected factor $F_0$, an $m$-tree-connected factor $F$, and 
 a $(0,l_0)$-partition-connected factor $\mathcal{F}$. Consider arbitrary orientations for $F$ and $F_0$, and consider an orientation for $\mathcal{F}$ such that for each vertex $v$, $d^-_{\mathcal{F}}(v)\ge l_0(v)$.
According to Lemma~\ref{lem:M}, for the case $m_0=0$, we may assume that $F$ include an arbitrary given factor of $G$ with maximum degree at most $m$.
Obviously, for each vertex $v$, $d_F(v)\ge m$ and $d_{F_0}(v)\ge m_0$.
Moreover, we also have
$ d_{G}^+(v)+d^-_{F}(v)+d^-_{F_0}(v) =d_G(v)-d^-_{\mathcal{F}}(v)\le d_G(v)-l_0(v)$.
Thus by applying Corollary~\ref{cor:list:F-F0} with $s=m$ and $s_0=m_0$,
 the graph $G$ has an $L$-factor $H$ including $F$ excluding $F_0$. Hence the theorem holds.
}\end{proof}
\subsection{Graphs with edge-connectivity at least $2m+2m_0$}
The following theorem gives a sufficient edge-connectivity condition for the existence of tree-connected factors 
with given dense lists on degrees.
\begin{thm}\label{thm:list:edge-connected}
{Let $G$ be a $(2m+2m_0)$-edge-connected graph.
For each vertex $v$, let $L(v)\subseteq \{m,\ldots, d_G(v)-m_0\}$.
Then $G$ has an $m$-tree-connected $L$-factor $H$ such that its complement is $m_0$-tree-connected, 
if for each vertex $v$, 
$$|L(v)|\ge \lceil\frac{d_G(v)}{2}\rceil+1.$$
}\end{thm}
\begin{proof}
{By Theorem~\ref{thm:basic:M}, 
the graph $G$ is $(m+m_0, l_0)$-partition-connected such that 
for each vertex $v$, 
$l_0(v)=\lfloor d_G(v)/2 \rfloor-m-m_0$.
By Theorem~\ref{thm:list:partition},
the graph $G$ has an $m$-tree-connected $L$-factor $H$ such that its complement is $m_0$-tree-connected.
}\end{proof}
\begin{cor}{\rm (\cite{Shirazi-Verstraete-2008})}
{Let $G$ be a graph and let $L(v)\subseteq \{0,\ldots, d_G(v)\}$ for each vertex $v$.
Then $G$ has an $L$-factor, if for each vertex $v$, 
$|L(v)|\ge \lceil\frac{d_G(v)}{2}\rceil+1.$
}\end{cor}
A strengthened version for Theorem 4 in~\cite{AHO} is given in the following result.
\begin{thm}\label{cor:list-tree-connected} 
{Let $G$ be a $2m$-edge-connected graph with two edge-disjoint loopless factors $M$ and $M_0$ satisfying $\Delta(M)\le m$ and $|E(M_0)|\le m$. Assume that $M_0$ is directed.
 For each vertex $v$, let $L(v)\subseteq \{m,\ldots, d_G(v)-d_{M_0}(v)\}$ and let $r(v)$ be a nonnegative integer satisfying $\sum_{v\in V(G)}r(v)=m-|E(M_0)|$.
 Then $G$ admits an $m$-tree-connected $L$-factor $H$ including $M$ excluding $M_0$, if for each vertex~$v$, 
$$|L(v)|\ge \lfloor\frac{d_G(v)}{2}\rfloor +1-r(v)-d^+_{M_0}(v).$$
Furthermore, for an arbitrary given vertex $z$, the lower bound can be reduced to $\lfloor\frac{d_G(z)}{2}\rfloor +1-d^+_{M_0}(z)$.
}\end{thm}
\begin{proof}
{Since $G$ is $2m$-edge-connected, by applying Corollary~\ref{cor:orientation},
$G$ has an $m$-tree-connected factor $F$ including $M$ excluding $M_0$ 
and the pre-orientation of $M_0$ can be extended to an orientation of $G$ 
 such that for each vertex $v$, $d^+_{G}(v)\le \lceil d_{G}(v)/2\rceil$, and $d^-_{F}(v)= m-r(v)-d^-_{M_0}(v)$.
Furthermore, we can have $d^+_{G}(z)\le \lfloor d_{G}(z)/2\rfloor$.
By Corollary~\ref{cor:list:F-F0} with setting $s(v)=m$ and $s_0(v)=d_{M_0}(v)$,
the graph $G$ must have an $L$-factor $H$ including $F$ excluding $M_0$.
Hence $H$ is the desired factor of $G$.
}\end{proof}

\section{Factors with bounded degrees}
\label{sec:bounded}
\subsection{Preliminary results}
Our aim in this section is to provide a fundamental result to prove the main theorems of this paper.
Before stating the result, let us recall a theorem due to Lov\'{a}sz (1970) who gave the following criterion for the existence of $(g,f)$-factors in general graphs.
\begin{thm}{\rm (\cite{Lovasz-1970})}\label{thm:Lovasz}
{Let $G$ be a graph and let $g$ and $f$ be two integer-valued functions on $V(G)$ with $g<f$.
Then $G$ has a $(g,f)$-factor if and only if for any two disjoint subsets $A$ and $B$ of $V(G)$, 
$$0\le \sum_{v\in A}f(v)+\sum_{v\in B}(d_G(v)-g(v))-d_G(A,B).$$
Furthermore, the theorem holds when $g(z)=f(z)$ for at most one vertex $z$.
}\end{thm}
Before stating the main result, let us derive the following generalization of Theorem~\ref{thm:Lovasz}.
\begin{cor}\label{cor:base}
{Let $G$ be a graph and let $g$ and $f$ be two integer-valued functions on $V(G)$ with $g< f$.
 Let $F$ and $F_0$ be two edge-disjoint factors of $G$.
Then $G$ has a $(g,f)$-factor $H$ including $F$ excluding $F_0$ if and only if for any two disjoint subsets $A$ and $B$ of $V(G)$, 
$$\sum_{v\in A}d_F(v)-d_F(A,B)+\sum_{v\in B}d_{F_0}(v)-d_{F_0}(A,B)\le 
 \sum_{v\in A}f(v)+\sum_{v\in B}(d_G(v)-g(v))-d_G(A,B).$$
Furthermore, the theorem holds when $g(z)=f(z)$ for at most one vertex $z$.
}\end{cor}
\begin{proof}
{Let $G'=G\setminus E(F\cup F_0)$.
Obviously, the graph $G$ has a $(g,f)$-factor $H$ including $F$ excluding $F_0$ if and only if $G'$ has a $(g',f')$-factor including $F$, where $g'(v)=g(v)-d_F(v)$ and $f'(v)=f(v)-d_F(v)$ for each vertex $v$. 
Note that the condition $g(v)<f(v)$ implies $g'(v)<f'(v)$.
Since $d_{G'}(v)=d_G(v)-d_F(v)-d_{F_0}(v)$ for each vertex $v$, and 
$d_{G'}(A,B)=d_{G}(A,B)-d_{F}(A,B)-d_{F_0}(A,B)$ for any two disjoint subsets $A$ and $B$ of $V(G)$, the assertion can be proved with respect to Theorem~\ref{thm:Lovasz}.
}\end{proof}
The following corollary is an interesting application of Corollary~\ref{cor:base} which plays a basic role in this paper.
\begin{cor}\label{cor:F-F0}
{Let $G$ be a graph and let $g$ and $f$ be two integer-valued functions on $V(G)$ with $g<f$.
 Let $F$ and $F_0$ be two edge-disjoint factors of $G$.
Then $G$ has a $(g,f)$-factor including $F$ excluding $F_0$, if $G$ has an orientation such that for each vertex $v$, 
$$g(v)\le d^+_G(v)+d^-_{F}(v)-d^+_{F_0}(v) \text{ and } d^-_G(v)+d^+_F(v)- d^-_{F_0}(v)\le f(v).$$
Furthermore, the assertion holds when $g(z)=f(z)$ for at most one vertex $z$.
}\end{cor}
\begin{proof}
{Define $G'$ to be the factor of $G$ with $E(G')=E(G)\setminus E(F\cup F_0)$.
Let $A$ and $B$ be two disjoint subsets of $V(G)$. It is easy to check that
$$\sum_{v\in A}d_F(v)=
\sum_{v\in A}(d^-_F(v)+d^+_F(v))\le 
\sum_{v\in A}(d_G^-(v)-d_{F_0}^-(v)+d_F^+(v))-d^-_{G'}(A,B),$$
where $d^-_{G'}(A,B)$ denotes the number of edges of $G$ whose heads are in $A$ and whose tails are in $B$.
Similarly, we must have
$$\sum_{v\in B}d_{F_0}(v)=
\sum_{v\in B}(d^-_{F_0}(v)+d^+_{F_0}(v))\le
 \sum_{v\in B}(d_G^-(v)-d_{F}^-(v)+d_{F_0}^+(v))-d^-_{G'}(B,A).$$
 Therefore,
$$\sum_{v\in A}d_F(v)+\sum_{v\in B}d_{F_0}(v)\le
\sum_{v\in A}(d_G^-(v)+d_F^+(v)-d_{F_0}^-(v))+
 \sum_{v\in B}(d_G(v)-(d_G^+(v)+d_{F}^-(v)-d_{F_0}^+(v))-d_{G'}(A,B),$$
which implies that 
$$\sum_{v\in A}d_F(v)+\sum_{v\in B}d_{F_0}(v)\le\sum_{v\in A}f(v)+
\sum_{v\in B}(d_G(v)-g(v)) -d_{G}(A,B)+d_{F}(A,B)+d_{F_0}(A,B).$$
Thus the assertion holds by Corollary~\ref{cor:base}.
}\end{proof}
Here we give an alternative proof for a weaker version of Corollary~\ref{cor:F-F0}.

\begin{thm}\label{Highly:Bounded:thm:main}
{Let $G$ be a directed graph with two edge-disjoint factors $F$ and $F_0$. For each vertex $v$, take $s(v)$ and $s_0(v)$ to be two nonnegative integers with $s(v)+s_0(v)\ge d_G^+(v)-d_G^-(v)$.
If $G$ is connected and $|E(F\cup F_0)|\ge 1$, then $G$ has a factor $H$ including $F$ excluding $F_0$ such that for each vertex $v$,
$$d^+_G(v)-d^+_{F_0}(v)-s_0(v)
\le d_{H}(v) \le d^-_G(v)+d^+_{F}(v)+s(v).$$
}\end{thm}
\begin{proof}
{Add some new directed edges to $G$ such that for each vertex $v$ of the resulting graph $\mathcal{G}$, we have
$d^+_\mathcal{G}(v)=d^-_\mathcal{G}(v)$, and also
$d_M(v)=|d^+_G(v) - d^-_G(v)|$, where $M$ is the factor of $\mathcal{G}$ consisting of all such new edges.
We may assume that $|E(F)|\ge 1$.
Since $\mathcal{G}$ is connected, it has a directed Eulerian tour $e_1, \ldots, e_t$
 by starting at an edge $e_1$ with $e_1\in E(F)$.
For a vertex $v$ with $d^+_G(v) > d^-_G(v)$,
let $\omega_j(v)$ be the $j$-th incoming edge of $v$ in the tour with
$\omega_j(v)=e_{i-1}\in E(M)$ and
$e_{i}\not \not \in E(F\cup F_0)$, where $e_{t+1}=e_1$ and $j\ge 1$. 
Set $W_v$ to be the set of all such edges.
Let $\mathcal{W}_v$ be the set of all edges $e_{i}$ incident with $v$ such that $e_{i-1}\in W_v$, where $i\ge 2$.
In plus, for all other vertices $v$ with $d^+_G(v) \le d^-_G(v)$, take $W_v$ and $\mathcal{W}_v$ to be the empty set.
Note that for each vertex $v$, we have $|\mathcal{W}_v|=|W_v|\le \max\{0, d^+_G(v)- d^-_G(v)\}\le s_1(v)+s_2(v)$.
Now, we shall construct the nested factors $H_0,\ldots, H_t$ of $G$.
Let $H_0$ be the graph with no edges. Take $i$ to be an integer with $1\le i\le t$. 
Now, with respect to the following rules inductively construct $H_i$ from $H_{i-1},$
\begin{enumerate}{
\item [{\upshape$\triangleright$}] 
 If $e_{i}\not \in E(M) \cup E(F\cup F_0)$,
	\begin{enumerate}{
 	
 	 \item [{\upshape $\triangleright$ }]
		If $e_{i-1} \in W_v$ with $e_{i-1}= \omega_j(v)$,
		\begin{enumerate}{
				\item [{\upshape $\triangleright$}]
					 If $j\le s_2(v)$,					\hspace{30mm} 	set $H_i = H_{i-1}$.

 	 			\item [{\upshape $\triangleright$}]
					Else if $j> s_2(v)$, 				\hspace{23mm} 	set $H_i = H_{i-1}+e_{i}$.
		}\end{enumerate}
		\item [{\upshape $\triangleright$}]
			Else 
			\begin{enumerate}{
	 			\item [{\upshape $\triangleright$}]
				If $e_{i-1}\in E(H_{i-1})$, 				\hspace{20.5mm} 	set $H_i = H_{i-1}$.
	 			 \item [{\upshape $\triangleright$}]
				Else if $e_{i-1}\not\in E(H_{i-1})$, 			\hspace{13mm} 	set $H_i = H_{i-1}+e_i$.
			}\end{enumerate}
	}\end{enumerate}
\item [{\upshape $\triangleright$}] 
 	Else if $e_{i}\in E(F)$,					\hspace{36.0mm} 	set $H_i = H_{i-1}+e_i $.
\item [{\upshape $\triangleright$}] 
	Else if $e_{i}\in E(M)\cup E(F_0)$, 			\hspace{22.0mm} 	set $H_i = H_{i-1}$.
}\end{enumerate}
Put $H=H_t$.
 Obviously, $H$ is a factor of $G$ including $F$ excluding $ F_0$.
Pick $v\in V(G)$. %
Let $E_H(v)$ be the set of all edges of $H$ incident with $v$, and
 let $Q(v)$ be the set of pair of edges of $G$ incident with $v$ as $e_j,e_{j+1}$ such that 
$ e_j\not\in W_v$ and $e_{j+1}\not\in E(F_1\cup F_2)\cup E(M)$.
Thus, we have
$$|Q(v)|/2=d^+_G(v)-d^+_{F}(v)-d^+_{F_0}(v)-|W_v|.$$
According to the construction of $H$, it is not difficult to verify that
$|E_H(v)\cap Q(v)|= |Q(v)|/2$, $|E_H(v)\cap E(F_0)|=0$, 
and $|E_H(v)\cap E(F)\setminus Q(v)|\ge d^+_{F}(v)$, 
$|E_H(v)\cap \mathcal{W}_v|\ge |\mathcal{W}_v|-s_2$.
Therefore, one can conclude that
$$d_H(v) \ge |Q(v)|/2+d_{F}^+(v)+|E_H(v)\cap \mathcal{W}_v|\ge d^+_G(v)-d^+_{F_0}(v)-s_0(v).$$
Also, we have
$d_H(v) \le d_G(v)-|Q(v)|/2-d_{F_0}^+(v)-\min\{s_0,|\mathcal{W}_v|\}\le 
 d^-_G(v)+d^+_{F}(v)+s(v).$
Hence the proof can be completed.
}\end{proof}
\subsection{Graphs with tree-connectivity at least $m+m_0$}
The following theorem gives a sufficient partition-connectivity condition for the existence of tree-connected factors 
with a given bound on degrees.
\begin{thm}\label{thm:partition-connected}
{Let $G$ be a graph and 
let $s$ and $l_0$ be two  integer-valued functions on $V(G)$ with $s\le l_0$.
If $G$ is $(m+m_0, l_0)$-partition-connected and $m+m_0>0$, then it admits
an $m$-tree-connected factor $H$ such that its complement is $m_0$-tree-connected and for each vertex~$v$,
$$m+s(v)\le d_{H}(v)\le \max\{m+s(v)+1,d_G(v)-l_0(v)-m_0\}.$$
Furthermore, for an arbitrary given vertex $z$, we can have $m+s(z)\le d_{H}(z)\le \max\{m+s(z),d_G(z)-l_0(z)-m_0\}.$
When $m_0=0$, $H$ can include an arbitrary given factor with maximum degree at most $m$.
}\end{thm}
\begin{proof}
{Decomposed $G$ into 
an $m_0$-tree-connected factor $F_0$, an $m$-tree-connected factor $F$, and 
 a $(0,l_0)$-partition-connected factor $\mathcal{F}$. Consider arbitrary orientations for $F$ and $F_0$, and consider an orientation for $\mathcal{F}$ such that for each vertex $v$, $d^+_{\mathcal{F}}(v)\ge l_0(v)$.
According to Lemma~\ref{lem:M}, for the case $m_0=0$, we may assume that $F$ includes an arbitrary given factor of $G$ with maximum degree at most $m$.
For each vertex $v$, define $g(v)=m+s(v)$ and 
$f(v)=\max\{m+s(v)+1, d_G(v)-l_0(v)-m_0\}$.
Obviously, $g(v)< f(v)$.
For the vertex $z$, we can define $f(z)=\max\{m+s(z), d_G(z)-l_0(z)-m_0\}$ so that $g(z)\le f(z)$.
Since $F$ is $m$-tree-connected, we must have $d_F(v)\ge m$ and so
$g(v)\le m+l_0(v)\le d_F(v)+d^+_{\mathcal{F}}(v)=d^+_G(v)+d^-_F(v)-d^+_{F_0}(v)$.
Since $F_0$ is $m_0$-tree-connected, we must have $d_{F_0}(v)\ge m_0$ and so
$d^-_G(v)+d^+_F(v)- d^-_{F_0}(v) =d_G(v)-d^+_{\mathcal{F}}(v)-d_{F_0}(v)\le d_G(v)-l_0(v)-m_0\le f(v)$.
Thus by Corollary~\ref{cor:F-F0},
 the graph $G$ has a $(g,f)$-factor $H$ including $F$ excluding $F_0$. Hence the theorem holds.
}\end{proof}
\subsection{Graphs with edge-connectivity at least $2m+2m_0$}
The following theorem gives a sufficient edge-connectivity condition for the existence of tree-connected factors
in graphs with a given bound on degrees.
\begin{thm}\label{thm:bounded:partition} 
{Let $G$ be a graph.
If $G$ is $(2m+2m_0)$-edge-connected and $m+m_0> 0$, then 
then $G$ has an $m$-tree-connected factor $H$ such that its complement is $m_0$-tree-connected and
 for each vertex $v$,
$$
\lfloor\frac{d_G(v)}{2}\rfloor-m_0\le d_{H}(v)\le \lceil\frac{d_G(v)}{2}\rceil+m.$$
Furthermore, for an arbitrary given vertex $z$, $d_H(z)$ can be assigned to any integer value in whose interval.
}\end{thm}
\begin{proof}
{By Theorem~\ref{thm:basic:M}, 
the graph $G$ is $(m+m_0, l_0)$-partition-connected such that 
for each $v\in V(G)\setminus \{z\}$, 
$l_0(v)=\lfloor d_G(v)/2 \rfloor-m-m_0$, and also $l_0(z)=\lceil d_G(z)/2 \rceil$.
By Theorem~\ref{thm:partition-connected},
the graph $G$ has an $m$-tree-connected $f$-factor $H$ such that its complement is $m_0$-tree-connected and
for each $v\in V(G)\setminus \{z\}$, 
$m+l_0(v)\le d_{H}(v)\le \max\{m+l_0(v)+1,d_G(v)-l_0(v)-m_0\}=d_G(v)-l_0(v)-m_0$, and also 
$m+s(z)\le d_{H}(z)\le \max\{m+s(z),d_G(z)-l_0(z)-m_0\}=m+s(z)$, where
$s(z)$ is an arbitrary integer with $\lfloor d_G(z)/2\rfloor-m_0-m \le s(z) \le \lceil d_G(z)/2\rceil=l_0(z)$.
Hence $H$ is the desired $f$-factor we are looking for.
}\end{proof}
A strengthened version for Theorem 8 in~\cite{AHO} is given in the following result.
\begin{thm}\label{thm:supplement:bounded}
{Let $G$ be a graph with $z\in V(G)$ and with two edge-disjoint loopless factors $M$ and $M_0$ satisfying $\Delta(M)\le m$ and $|E(M_0)|\le m$. Assume that $d_{M_0}(v)<m$ for all vertices $v$, except possibly for $z$.
If $G$ is $2m$-edge-connected and $m>0$, then it has an $m$-tree-connected factor $H$ including $M$ excluding $M_0$ such that for each vertex $v$, $$\lfloor\frac{d_G(v)}{2}\rfloor\le d_H(v) \le \lceil\frac{d_G(v)}{2}\rceil+m-d_{M_0}(v).$$
Furthermore, for the vertex $z$, the upper bound can be reduced to $\lfloor\frac{d_G(z)}{2}\rfloor +1-d_{M_0}(z)$.
}\end{thm}
\begin{proof}
{By applying Theorem~\ref{thm:basic:M}, one can conclude that
the graph $G\setminus E(M_0)$ is $(m, l_0)$-partition-connected such that 
for each $v\in V(G)\setminus \{z\}$, 
$l_0(v)=\lfloor d_G(v)/2 \rfloor-m$, and also
$l_0(z)=\lceil d_G(z)/2 \rceil-m$.
By Theorem~\ref{thm:list:partition},
the graph $G\setminus E(M_0)$ has an $m$-tree-connected factor $H$ including $M$ such that for each $v\in V(G)\setminus \{z\}$, 
$\lfloor d_G(v)/2\rfloor\le m+l_0(v)= d_{H}(v)\le \max\{m+l_0(v)+1,d_{G}(v)-d_{M_0}(v)-l_0(v)\}= 
\lceil d_G(v)/2\rceil+m-d_{M_0}(v)$, and also
$\lfloor d_G(z)/2\rfloor= m+s(z)\le d_{H}(z)\le \max\{m+l_0(z),d_{G}(z)-d_{M_0}(z)-l_0(z)\}= 
\lceil d_G(z)/2\rceil+m-d_{M_0}(z)$, where $s(z)=\lfloor d_G(z)/2\rfloor-m$.
Hence $H$ is the desired factor of $G$.
}\end{proof}
\begin{cor}
{Every connected $4$-regular graph has a connected $\{2,3\}$-factor including an arbitrary given matching.
}\end{cor}
\begin{proof}
{Apply Theorem~\ref{thm:supplement:bounded} with $m=1$.
}\end{proof}
%
%
%
%
%
%
%
%
\section{Factors modulo $2$}
\subsection{Preliminary results}
In this section, we consider the existence of connected $f$-factors modulo $2$ with bounded degrees in general graphs.
For this purpose, we need the following recent result on the existence of $f$-factors modulo $2$ with restricted degrees.
\begin{thm}{\rm (\cite{ModuloFactorBounded})}\label{thm:2:partition}
{Let $G$ be a graph and let $f:V(G)\rightarrow Z_2$ be a mapping with 
$ \sum_{v\in V(G)}f(v)$ even. 
Let $s$, $s_0$, and $l_0$ be three integer-valued functions on $V(G)$ satisfying $s+s_0< d_G$ and $\max\{s,s_0\}\le l_0$. 
If $G$ is $(1, l_0)$-partition-connected, then it admits an $f$-factor $H$ such that for each vertex~$v$, 
$$s(v)\le d_{H}(v)\le d_G(v)-s_0(v).$$
}\end{thm}
Before stating the main result, let us derive the following generalization of Theorem~\ref{thm:2:partition}.
\begin{cor}\label{cor:modulo-2:partition-connected}
{Let $G$ be a graph and let $f:V(G)\rightarrow Z_2$ be a mapping with 
$\sum_{v\in V(G)} f(v)\stackrel{2}{\equiv}0$.
Let $F$ and $F_0$ be two edge-disjoint factors of $G$, and 
let $s$, $s_0$, and $l_0$ be three integer-valued functions on $V(G)$ satisfying $s+s_0< d_G-d_F-d_{F_0}$ and
$\max\{s,s_0\}\le l_0$.
If $G\setminus E(F\cup F_0)$ is $(1, l_0)$-partition-connected, then $G$ admits
an $f$-factor $H$ including $F$ excluding $F_0$ such that for each vertex $v$,
$$d_{F}(v)+s(v)\le d_{H}(v)\le d_G(v)-s_0(v)-d_{F_0}(v).$$
}\end{cor}
\begin{proof}
{Let $G'=G\setminus E(F\cup F_0)$. For each vertex $v$, define $f'(v)=f(v)-d_{F}(v)$ (mod $2$).
Obviously, $\sum_{v\in V(G)} f'(v)\stackrel{2}{\equiv}0$.
Thus by Theorem~\ref{thm:2:partition}, the graph $G'$ has an $f'$-factor $H'$ such that 
for each vertex~$v$,
$s(v)\le d_{H'}(v)\le d_{G'}(v)-s_0(v).$
Hence $H'\cup F$ is the desired $f$-factor we are looking for.
}\end{proof}
\subsection{Graphs with tree-connectivity at least $m+m_0+1$}
The following theorem gives a sufficient partition-connectivity condition for the existence of tree-connected $f$-factors modulo $2$ 
with a given bound on degrees.
\begin{thm}\label{thm:modulo-2:partition-connected}
{Let $G$ be a graph and let $f:V(G)\rightarrow Z_2$ be a mapping satisfying $\sum_{v\in V(G)}f(v)\stackrel{2}{\equiv}0$ and 
let $s$ and $l_0$ be two  integer-valued functions on $V(G)$ with $s\le l_0$.
If $G$ is $(m+m_0+1, l_0)$-partition-connected, then it admits
an $m$-tree-connected $f$-factor $H$ such that its complement is $m_0$-tree-connected and for each vertex~$v$,
$$m+s(v)\le d_{H}(v)\le \max\{m+s(v)+1,d_G(v)-l_0(v)-m_0\}.$$
}\end{thm}
\begin{proof}
{Decomposed $G$ into 
an $m$-tree-connected factor $F$, an $m_0$-tree-connected factor $F_0$, and 
 a $(1,l_0)$-partition-connected factor $\mathcal{F}$.
For each vertex $v$, define $s'(v)=s(v)+m-d_F(v)$ and
$s'_0(v)=\min\{ d_G(v)-m-s(v)-1, l_0(v)+m_0\}-d_{F_0}(v)$.
Obviously, $s'(v)+s'_0(v)< d_G(v)-d_F(v)-d_{F_0}(v)$.
Since $d_F(v)\ge m$ and $d_{F_0}(v)\ge m_0$, we must also have $s'(v)\le s(v)\le l_0(v)$ and $s_0'(v)\le l_0(v)$.
Thus by Corollary~\ref{cor:modulo-2:partition-connected},
 the graph $G$ has an $f$-factor $H$ including $F$ excluding $F_0$ such that 
for each vertex~$v$,
$m+s(v)= d_F(v)+s'(v)\le d_{H}(v)\le d_G(v)-s'_0(v)-d_{F_0}(v)=
\max\{m+s(v)+1,d_G(v)-l_0(v)-m_0\}$.
Hence $H$ is the desired $f$-factor we are looking for.
}\end{proof}
\begin{cor}\label{thm:f-f+2}
{Let $G$ be a graph and let $f$ be an integer-valued function on $V(G)$ with $f\ge 2$ satisfying $\sum_{v\in V(G)}f(v)\stackrel{2}{\equiv}0$. If $G$ contains a $(2,f-2)$-partition-connected factor $G'$ satisfying $d_{G'}(v)\le 2f(v)+1$ for each vertex $v$,
 then $G$ admits a connected $\{f,f+2\}$-factor.
}\end{cor}
\begin{proof}
{By applying Theorem~\ref{thm:modulo-2:partition-connected} with $m=1$, $m_0=0$, and $s_0=l_0=f-2$, the graph $G'$ has a connected $f$-factor $H$ such that for each vertex $v$, $f(v) -1\le d_H(v)\le f(v)+3$. Hence $H$ is a connected $\{f,f+2\}$-factor of $G$.
}\end{proof}
\subsection{Graphs with edge-connectivity at least $2m+2m_0+2$}
The following theorem gives a sufficient edge-connectivity condition for the existence of tree-connected $f$-factors modulo $2$ 
with a given bound on degrees.
\begin{thm}\label{thm:modulo2:main}
{Let $G$ be a graph and let $f:V(G)\rightarrow Z_2$ be a mapping with
$\sum_{v\in V(G)}f(v) \stackrel{2}{\equiv} 0$. 
If $G$ is $(2m+2m_0+2)$-edge-connected, then it has an $m$-tree-connected $f$-factor $H$ such that its complement is
 $m_0$-tree-connected and for each vertex~$v$,
$$\lfloor \frac{d_{G}(v)}{2}\rfloor-1-m_0\le d_H(v)\le \lceil \frac{d_{G}(v)}{2}\rceil+1+m.$$
Furthermore, for an arbitrary given vertex $z$, $d_H(z)$ can be assigned to any plausible integer value in whose interval.
}\end{thm}
\begin{proof}
{By Theorem~\ref{thm:basic:M}, 
the graph $G$ is $(m+m_0+1, l_0)$-partition-connected such that 
for each $v\in V(G)\setminus \{z\}$, 
$l_0(v)=\lfloor d_G(v)/2 \rfloor-m-m_0-1$, and also $l_0(z)=\lceil d_G(z)/2 \rceil$. 
By Theorem~\ref{thm:modulo-2:partition-connected},
the graph $G$ has an $m$-tree-connected $f$-factor $H$ such that its complement is $m_0$-tree-connected and for each $v\in V(G)\setminus \{z\}$, 
 $\lfloor d_G(v)/2\rfloor-1-m_0 =m+l_0(v)\le d_{H}(v)\le \max\{m+l_0(v)+1,d_G(v)-l_0(v)-m_0\}=d_G(v)-l_0(v)-m_0 =
\lceil d_G(v)/2\rceil+1+m$, 
and also
$m+s(z)\le d_{H}(z)\le \max\{m+s(z)+1,d_G(z)-l_0(z)-m_0\}
=m+s(z)+1$, where $s(z)$ is an arbitrary integer 
with $ \lfloor d_G(z)/2\rfloor-m_0-m\le s(z)< \lceil d_G(z)/2 \rceil= l_0(z)$.
Hence $H$ is the desired $f$-factor we are looking for.
}\end{proof}
\begin{cor}
{Every $6$-edge-connected $2r$-regular graph $G$ with $(r+1)|V(G)|$ even can be decomposed into two connected $\{r-1,r+1\}$-factors.
}\end{cor}
\begin{proof}
{Apply Theorem~\ref{thm:modulo2:main} with $m=m_0=1$ and $f=r-1$ (mod $2$).
}\end{proof}
The following corollary improves a result of Jaeger (1979)~\cite{Jaeger-1979}, on the existence of spanning Eulerian subgraphs in $4$-edge-connected graphs, together with a generalization of it due to Catlin (1988)~\cite{Catlin-1988}.
\begin{cor}\label{cor:k=2:tree-connected}
{Let $G$ be a graph and let $f:V(G)\rightarrow Z_2$ be a mapping with
$\sum_{v\in V(G)}f(v) \stackrel{2}{\equiv} 0$. 
If $G$ is $4$-edge-connected, then it has a connected $f$-factor $H$ such that for each vertex~$v$,
$$\lfloor \frac{d_{G}(v)}{2}\rfloor-1\le d_H(v)\le \lceil \frac{d_{G}(v)}{2}\rceil+2.$$
}\end{cor}
\begin{proof}
{Apply Theorem~\ref{thm:modulo2:main} with $m=1$ and $m_0=0$.
}\end{proof}
\begin{cor}
{Every $4$-edge-connected graph $G$ 
 has a connected factor $H$ such that for each vertex~$v$,
$\lfloor d_G(v)/4\rfloor\le d_H(v) \le \lceil (d_G(v)-2)/4\rceil+2.$
}\end{cor}
\begin{proof}
{By Corollary~\ref{cor:k=2:tree-connected}, the graph $G$ has a connected factor $H_0$ with even degrees
 such that for each vertex~$v$,
$(d_G(v)+1)/2-2\le d_{H_0}(v) \le (d_G(v)+1)/2+2$.
Since $H_0$ is $2$-edge-connected, by Theorem~\ref{thm:bounded:partition},
the graph $H_0$ has a connected factor $H$ such that for each vertex $v$, 
$d_{H_0}(v)/2\le d_H(v) \le d_{H_0}(v)/2+1.$
These inequalities can complete the proof.
}\end{proof}
\begin{cor}
{Every $4$-edge-connected $2r$-regular graph $G$ with $(r+1)|V(G)|$ even admits a connected $\{r-1,r+1\}$-factor.
}\end{cor}
\begin{proof}
{Apply Corollary~\ref{cor:k=2:tree-connected} with $f=r-1$ (mod $2$).
}\end{proof}
The following corollary provides a minor refinement for Theorem 9 in \cite{Hasanvand-2015}.
\begin{cor}
{Every $4$-edge-connected $r$-regular graph has a connected $\{4,6\}$-factor, where $8\le r\le 10$.
}\end{cor}
\begin{proof}
{Apply Corollary~\ref{cor:k=2:tree-connected} with $f=0$ (mod $2$).
}\end{proof}
%
%
%
%
%
%
%
%
%
\section{Factors modulo $k$: bipartite graphs}
\label{sec:bipartite}
\subsection{Preliminary results}
In this section, we consider the existence of connected $f$-factors modulo $k$ with bounded degrees in bipartite graphs.
For this purpose, we need the following recent result on the existence $f$-factors modulo $k$ with restricted degrees.
\begin{thm}{\rm (\cite{ModuloFactorBounded})}\label{thm:bi:k:partition}
{Let $G$ be a bipartite graph with $z\in V(G)$, let $k$ be an integer, $k\ge 2$, and let $f:V(G)\rightarrow Z_k$ be a compatible mapping.
Let $s$, $s_0$, and $l_0$ be three integer-valued functions on $V(G)$ satisfying $s+s_0+k-1\le d_G$ and
$\max\{s,s_0\}\le l_0+(k-1)\bar{\chi}_z$. 
If $G$ is $(2k-2, l_0)$-partition-connected, then it admits
an $f$-factor $H$ such that for each vertex $v$,
$$s(v) \le d_H(v) \le d_G(v)-s_0(v).$$
}\end{thm}
Before stating the main result, let us derive the following generalization of Theorem~\ref{thm:bi:k:partition}.
\begin{cor}\label{cor:modulo-k:partition-connected}
{Let $G$ be a bipartite graph with $z\in V(G)$, let $k$ be an integer, $k\ge 2$, and let $f:V(G)\rightarrow Z_k$ be a compatible mapping.
Let $F$ and $F_0$ be two edge-disjoint factors of $G$, and
let $s$, $s_0$, and $l_0$ be three integer-valued functions on $V(G)$ satisfying 
 $s+s_0+k-1\le d_G-d_{F}-d_{F_0}$ and $\max\{s,s_0\} \le l_0+(k-1)\bar{\chi}_z$.
If $G\setminus E(F\cup F_0)$ is $(2k-2, l_0)$-partition-connected, then $G$ admits
an $f$-factor $H$ including $F$ excluding $F_0$ such that for each vertex $v$,
$$d_{F}(v)+s(v)\le d_{H}(v)\le d_G(v)-s_0(v)-d_{F_0}(v).$$
}\end{cor}
\begin{proof}
{Let $G'=G- E(F\cup F_0)$. For each vertex $v$, define $f'(v)=f(v)-d_{F}(v)$ (mod $k$).
Obviously, $f'$ is compatibly with $G$.
Thus by Theorem~\ref{thm:bi:k:partition}, the graph $G'$ has an $f'$-factor $H'$ such that 
for each vertex $v$,
$s(v) \le d_{H'}(v) \le d_{G'}(v)-s_0(v)$.
Hence $H'\cup F$ is the desired $f$-factor we are looking for.
}\end{proof}
\subsection{Graphs with tree-connectivity at least $m+m_0+2k-2$}
The following theorem gives a sufficient partition-connectivity condition for the existence of tree-connected $f$-factors modulo $k$ 
in bipartite graphs with a given bound on degrees.
\begin{thm}\label{thm:bipartite:partition-connected}
{Let $G$ be a bipartite graph with $z\in V(G)$, let $k$ be an integer, $k\ge 2$, and let $f:V(G)\rightarrow Z_k$ be a compatible mapping.
Let $s$ and $l_0$ be two integer-valued functions on $V(G)$ with $s \le l_0+(k-1)\bar{\chi}_z$.
If $G$ is $(m+m_0+2k-2, l_0)$-partition-connected, then it admits
an $m$-tree-connected $f$-factor $H$ such that its complement is $m_0$-tree-connected and for each vertex $v$,
$$m+s(v) \le d_H(v) \le \max\{m+s(v)+k-1, d_G(v)-l_0(v)-(k-1)\bar{\chi}_z(v)-m_0\}.$$
}\end{thm}
\begin{proof}
{Decomposed $G$ into 
an $m$-tree-connected factor $F$, an $m_0$-tree-connected factor $F_0$, and 
 a $(2k-2,l_0)$-partition-connected factor $\mathcal{F}$.
For each vertex $v$, define $s'(v)=s(v)+m-d_F(v)$ and 
$s'_0(v)=\min\{ d_G(v)-m-s(v)-(k-1),s_0(v)+m_0\}-d_{F_0}(v)$, where $s_0(v)=l_0(v)+(k-1)\bar{\chi}_z(v)$.
Obviously, $s'(v)+s'_0(v)+k-1\le d_G(v)-d_F(v)-d_{F_0}(v)$.
Since $d_F(v)\ge m$ and $d_{F_0}(v)\ge m_0$, 
we must also have 
$s'(v)\le s(v)\le l_0(v)+(k-1)\bar{\chi}_z(v)$ and 
$s_0'(v)\le s_0(v)= l_0(v)+(k-1)\bar{\chi}_z(v)$.
Thus by Corollary~\ref{cor:modulo-k:partition-connected},
 the graph $G$ has an $f$-factor $H$ including $F$ excluding $F_0$ such that 
for each vertex~$v$,
$m+s(v)= d_F(v)+s'(v)\le d_{H}(v)\le d_G(v)-s'_0(v)-d_{F_0}(v)=
\max\{m+s(v)+k-1,d_G(v)-s_0(v)-m_0\}$.
Hence $H$ is the desired $f$-factor we are looking for.
}\end{proof}
\subsection{Graphs with edge-connectivity at least $2m+2m_0+4k-4$}
The following theorem gives a sufficient edge-connectivity condition for the existence of tree-connected $f$-factors modulo $k$ 
in bipartite graphs with a given bound on degrees.
\begin{thm}\label{thm:modulo:bi:first}
{Let $G$ be a bipartite graph, let $k$ be an integer, $k\ge 2$, and let $f:V(G)\rightarrow Z_k$ be a compatible mapping. 
If $G$ is $(2m+2m_0+4k-4)$-edge-connected, then it has an $m$-tree-connected $f$-factor $H$ such that its complement is
 $m_0$-tree-connected and for each vertex~$v$,
$$\lfloor \frac{d_G(v)}{2}\rfloor -(k-1) -m_0\le d_H(v) \le \lceil \frac{d_G(v)}{2}\rceil+(k-1)+m.$$
Furthermore, for an arbitrary given vertex $z$, $d_H(z)$ can be assigned to any plausible integer value in whose interval.
}\end{thm}
\begin{proof}
{By Theorem~\ref{thm:basic:M}, 
the graph $G$ is $(m+m_0+2k-2, l_0)$-partition-connected such that 
for each $v\in V(G)\setminus \{z\}$, 
$l_0(v)=\lfloor d_G(v)/2 \rfloor-m-m_0-(2k-2)$, and also $l_0(z)=\lceil d_G(z)/2 \rceil$.
By applying Theorem~\ref{thm:bipartite:partition-connected},
the graph $G$ has an $m$-tree-connected $f$-factor $H$ such that its complement is $m_0$-tree-connected and for each $v\in V(G)\setminus \{z\}$, 
$\lfloor d_G(v)/2\rfloor -(k-1) -m_0 = 
m+s(v)\le d_{H}(v)\le \max\{m+s(v)+k-1,d_G(v)-l_0(v)-(k-1)-m_0\}=
d_G(v)-l_0(v)-(k-1)-m_0=\lceil d_G(v)/2\rceil+(k-1)+m$, 
where $s(v)=l_0(v)+k-1$,
 and also
$m+s(z)\le d_{H}(z)\le \max\{m+s(z)+k-1,d_G(z)-l_0(z)-m_0\}
=m+s(z)+k-1$, where $s(z)$ is an arbitrary integer 
with $ \lfloor d_G(z)/2\rfloor-(k-1)-m_0-m\le s(z)\le l_0(z)=\lceil d_G(z)/2 \rceil$.
Hence $H$ is the desired $f$-factor we are looking for.
}\end{proof}
\begin{cor}
{Every bipartite $(4k-2)$-edge-connected Eulerian graph $G$
has a connected factor $H$ such that for each vertex~$v$,
$d_H(v)\in \{d_G(v)/2, d_G(v)/2+k\}.$ 
}\end{cor}
\begin{proof}
{Let $(X,Y)$ be the bipartition of $G$. Since $\sum_{v\in X}d_G(v)/2 =\sum_{v\in Y} d_G(v)/2$, the mapping $f(v)=d_G(v)/2$ (mod $k$) is compatible with $G$. By applying Theorem~\ref{thm:modulo:bi:first} with $m=1$, the graph $G$ has a connected $f$-factor $H$ such that for each vertex $v$, $d_G(v)/2-k< d_H(v)\le d_G(v)/2+k$. Hence the assertion holds.
}\end{proof}
\begin{cor}
{Every bipartite $8k$-edge-connected Eulerian graph $G$ of even order can be decomposed into two
 connected factors $G_1$ and $G_2$ such that for each vertex $v$, $$d_{G_1}(v),d_{G_2}(v)\in \{d_G(v)/2-k, d_G(v)/2+k\}.$$ 
}\end{cor}
\begin{proof}
{Let $(X,Y)$ be the bipartition of $G$. Since $G$ has even order and $\sum_{v\in X}d_G(v)/2 =\sum_{v\in Y} d_G(v)/2$, the mapping $f(v)=d_G(v)/2+k$ (mod $2k$) must be compatible with $G$. By applying Theorem~\ref{thm:modulo:bi:first} with $m=m_0=1$, the graph $G$ has a connected $f$-factor $H$ such that its complement is connected and for each vertex $v$, $d_G(v)/2-2k\le d_H(v)\le d_G(v)/2+2k$. This can complete the proof.
}\end{proof}
\subsection{A restriction on degrees of one partite set}
Our aim in this subsection is to improve the needed tree-connectivity of Lemma 4.3 in \cite{Merker-2017}.
Our proof is based on the following lemma which is an improved version of Lemma 4.1 in \cite{Merker-2017}.
\begin{lem}{\rm (\cite{ModuloFactorBounded})}\label{lem:modulo:1/2}
{Let $G$ be a bipartite multigraph with bipartition $(X,Y)$ in which each vertex in $X$ has even degree.
Let $f$ be an integer-valued function on $Y$ satisfying $\sum_{v\in Y}f(v)\stackrel{k}{\equiv}\frac{1}{2}|E(G)|$.
 If $G$ is $(3k-3)$-edge-connected, then
there exists a factor $H$ of $G$ such that 
\begin{enumerate}{
\item For each $v\in X$, $d_{H}(v)=\frac{1}{2}d_G(v)$.
\item  For each $v\in Y$, $d_H(v)\stackrel{k}{\equiv} f(v)$ and $|d_H(v)- \frac{1}{2}d_G(v)|<k$.
}\end{enumerate}
}\end{lem}
Before stating the main result, let us formulate the following development of Lemma~\ref{lem:modulo:1/2}.
\begin{thm}\label{thm:modulo:1/2:F-F0}
{Let $G$ be a bipartite multigraph with bipartition $(X,Y)$ and let  $f$ be an integer-valued function on $V(G)$ satisfying $\sum_{v\in Y}f(v)\stackrel{k}{\equiv}\sum_{v\in X}f(v)$.
Let $F$, $F_0$, and $T$ be three edge-disjoint factors of $G$ such that for each $v\in X$, 
 $$d_F(v)+d_T(v)/2\le f(v)\le d_G(v)-d_{F_0}(v)-d_T(v)/2.$$ 
 If $T$ is $(3k-3)$-edge-connected, then
there exists a factor $H$ of $G$ including $F$ excluding $F_0$ such that 
\begin{enumerate}{
\item  For each $v\in X$, $d_{H}(v)=f(v)$.
\item  For each $v\in Y$, $d_H(v)\stackrel{k}{\equiv} f(v)$, and
$\frac{1}{2}d_T(v)-k<\min \{d_H(v)-d_F(v),d_{H_0}(v)-d_{F_0}(v)\}$, where $H_0=G\setminus E(H)$.
}\end{enumerate}
}\end{thm}
\begin{proof}
{For convenience, let us define $G_0=G\setminus E(F_0)$.
Since for each $v\in X$, $d_F(v)\le f(v)-d_T(v)/2$, without loss of generality,
we can assume that 
$d_F(v)= \min\{\lfloor f(v)-d_T(v)/2\rfloor, d_{G_0}(v)-d_T(v)\}$ by adding some of the edges of $E(G_0)\setminus E(T\cup F)$ to $F$ (if necessary). 
Since $f(v)\le d_{G_0}(v)-d_T(v)/2$, we must have 
$2f(v)-d_{G_0}(v)\le \min \{f(v)-d_T(v)/2, d_{G_0}(v)-d_T(v)\}$
which implies that $2f(v)-d_{G_0}(v) \le d_F(v)$.
Therefore, for each $v\in X$, $d_T(v)\le 2f(v)-2d_F(v) \le d_{G_0}(v)-d_F(v)$.
This can imply that there exists a factor $G'$ of $G_0\setminus E(F)$ including $T$
such that for $v\in X$, $d_{G'}(v)= 2f(v)-2d_F(v)$.
 For each vertex $v$, let $f'(v)=f(v)-d_F(v)$.
According to the assumption, we must have 
$\sum_{v\in Y}f'(v)= \sum_{v\in Y}(f(v)-d_F(v))\stackrel{k}{\equiv}\sum_{v\in X}(f(v)-d_F(v))=\frac{1}{2}|E(G')|$. 
Since $G'$ is $(3k-3)$-edge-connected, by Lemma~\ref{lem:modulo:1/2}, there exists a factor $H'$ of $G'$ such that 
for each $v\in X$, $d_{H'}(v)=\frac{1}{2}d_{G'}(v)=f'(v)$, and 
 for each $v\in Y$, $d_{H'}(v)\stackrel{k}{\equiv} f'(v)$ and $|d_{H'}(v)- \frac{1}{2}d_{G'}(v)|<k$.
It is not hard to check that $ H' \cup F$ is the desired factor we are looking for.
}\end{proof}
\begin{lem}{\rm (\cite{ClosedTrails})}\label{lem:small-degree}
{Let $G$ be a graph with an independent set $X\subseteq V(G)$. 
If $G$ is $k$-tree-connected and $k\ge m$, then it has an $m$-tree-connected factor $H$ 
such that for each~$v\in X$, $d_H(v)\le \lceil \frac{m}{k} d_G(v)\rceil$.
}\end{lem} 
The following result
 can improve the needed tree-connectivity of Lemma 4.3 and Propositions 4.2 and 5.2 in~\cite{Merker-2017} simultaneously.
In particular, it reduces the needed tree-connectivity of Lemma 4.3 in \cite{Merker-2017}, with the special case $m_0=1$,
by replacing the  linear upper bound of $2\lambda m+3km$ instead of the former quadratic bound of $8\lambda m^2+12km$.
\begin{cor}\label{thm:m,m0,k,lamda, lambda0}
{Let $\lambda$ and $\lambda_0$ be two nonengative integers and let $k$, $m$, and $m_0$ be three positive integers with $m\ge 2m_0>0$. 
Let $G$ be a bipartite graph with bipartition $(X,Y)$ in which each vertex in $X$ has degree divisible by $m$ and 
and let $f$ be an integer-valued function on $Y$ satisfying $\sum_{v\in Y}f(v)\stackrel{k}{\equiv}\frac{m_0}{m}|E(G)|$. If $G$ has tree-connectivity at least  $$\lceil m(\frac{\lambda}{m_0}+\frac{\lambda_0}{m-m_0}+\frac{3k}{2m_0})\rceil,$$ then
there exists a $\lambda$-tree-connected factor $H$ of $G$ such that the complement of it is $\lambda_0$-tree-connected, and
\begin{enumerate}{
\item For each $v\in X$, $d_{H}(v)=\frac{m_0}{m}d_G(v)$.
\item  For each $v\in Y$, $d_H(v)\stackrel{k}{\equiv} f(v)$.
}\end{enumerate}
}\end{cor}
\begin{proof}
{Decompose $G$ into three factors $G_1$, $G_2$, and $G_3$ such that 
$G_1$ is $\lceil \frac{m\lambda}{m_0}\rceil$-tree-connected, and $G_2$
is $\lceil \frac{\lambda_0}{m-m_0}\rceil$-tree-connected, and $G_3$ is
$\lceil \frac{(3k-2)m}{2m_0}\rceil$-tree-connected.
According to Lemma~\ref{lem:small-degree}, 
the graph $G_1$ has a $\lambda$-tree-connected factor $F$, 
the graph $G_2$ has a $\lambda_0$-tree-connected factor $F_0$, and
the graph $G_3$ has a $(3k-2)$-tree-connected factor $T'$ 
such that for each $v\in X$, 
 $d_{F}(v)\le \lceil \frac{m_0}{m} d_{G_1}(v)\rceil$, 
$d_{F_0}(v)\le \lceil \frac{m-m_0}{m} d_{G_2}(v)\rceil$, and 
$d_{T'}(v)\le \lceil \frac{2m_0}{m} d_{G_3}(v)\rceil$.
Let $T$ be a $(3k-3)$-tree-connected factor of $T$ such that for each $v\in X$,
 $d_T(v)\le d_{T'}(v)-1\le \frac{2m_0}{m} d_{G_3}(v)$.
For each $v\in X$, define $f(v)=\frac{m_0}{m}d_G(v)$.
Then $d_F(v)+d_T(v)/2\le \lceil \frac{m_0}{m} (d_{G_1}(v)+d_{G_2}(v))\rceil \le \frac{m_0}{m} d_G(v)=f(v)$.
Since $m-m_0\ge m_0$, we must similarly have $f(v)\le d_G(v)-d_{F_0}(v)-d_T(v)/2$.
According to the assumption, $\sum_{v\in Y}f(v)\stackrel{k}{\equiv}\frac{m_0}{m}|E(G)|=\sum_{v\in X}f(v)$.
Therefore, by Theorem~\ref{thm:modulo:1/2:F-F0}, the graph $G$ has a factor $H$ including $F$ excluding $F_0$ such that 
for each $v\in X$, $d_{H}(v)=f(v)$, and for each $v\in Y$, $d_H(v)\stackrel{k}{\equiv} f(v)$.
This can complete the proof.
}\end{proof}
\section{Factors modulo $k$: general graphs}
\label{sec:general-graphs}
\subsection{Preliminary results}
In this section, we are going to develop Theorem~\ref{thm:modulo:bi:first} to general graphs.
For this purpose, we need to make some basic results for working with general graphs using bipartite factors.
\begin{lem}{\rm (\cite{ModuloFactorBounded})}
\label{lem:tree-connected:bipartite}\label{lem:bipartite:factor}
{Every $2m$-tree-connected graph has an $m$-tree-connected bipartite factor.
}\end{lem}
The following theorem is useful for finding tree-connected factors with given bipartite index.
\begin{thm}\label{thm:decomposition:bipartite-index}
{If $G$ is an $(m_1+2m_2)$-tree-connected graph and $m_2\ge k_0\ge 0$, then $G$ can be decomposed into two factors $G_1$ and $G_2$ such that $G_1$ is $m_1$-tree-connected, $G_2[X,Y]$ is $m_2$-tree-connected for a bipartition $X,Y$ of $V(G)$, and
$$e_{G_2}(X)+e_{G_2}(Y)= \min\{k_0, bi(G)\}.$$
}\end{thm}
\begin{proof}
{Decompose $G$ into two factors $H_1$ and $H_2$ such that $H_1$ is $m_1$-tree-connected and $H_2$ is $2m_2$-tree-connected. By Lemma~\ref{lem:bipartite:factor}, there is a bipartition $X,Y$ of $V(G)$ such that $H_2[X,Y]$ is $m_2$-tree-connected. Let $t=e_{H_2}(X)+e_{H_2}(Y)$ and $b=\min\{k_0, bi(G)\}$. If $t> b$, we remove some of the edges of $H_2[X]\cup H_2[Y]$ and add them to $H_1$. Call the resulting graphs $G_1$ and $G_2$. Obviously, this can be done such that $e_{G_2}(X)+e_{G_2}(Y)=b$. If $t= b$, we set $G_1=H_1$ and $G_2=H_2$. If $t< b$, then $H_2[X,Y]$ must clearly be $(2m_2-t)$-tree-connected. Since $m_2\ge b$, we can decompose $H_2[X,Y]$ into two factors $F_1$ and $F_2$ such that $F_1$ is $(b-t)$-tree-connected and $F_2$ is $m_2$-tree-connected. Let $M$ be a factor of $H_1[X]\cup H_1[Y]$ with size $b-t$.
Now, define $G_1=H_1-E(M)+E(F_1)$ and $G_2=H_2-E(F_1)+E(M)$. It is easy to check that $G_1$ and $G_2$ are the desired factors. 
}\end{proof}
The following corollary plays an important role in this section.
\begin{cor}\label{cor:decomposition:bipartite-index}
{If $G$ is an $(m_1+2m_2+1)$-tree-connected graph and $m_2\ge k_0\ge 0$, then $G$ can be decomposed into two factors $G_1$ and $G_2$ such that $G_1$ is $m_1$-tree-connected, $G_2[X,Y]$ is $m_2$-tree-connected for a bipartition $X,Y$ of $V(G)$, $G_i$ whose vertex degrees are even for an arbitrary given integer $i\in \{1,2\}$, and
$$e_{G_2}(X)+e_{G_2}(Y)\ge \min\{k_0, bi(G)\}.$$ 
 Furthermore, the equality can hold when $G$ is $(m_1+2m_2+2)$-tree-connected.
}\end{cor}
\begin{proof}
{By Theorem~\ref{thm:decomposition:bipartite-index}, one can conclude that the graph $G$ can be decomposed into three factors $T$, $G'_1$, and $G'_2$
 such that $T$ is connected, $G'_1$ is $m_1$-tree-connected, $G'_2[X,Y]$ is $m_2$-tree-connected for a bipartition $X,Y$ of $V(G)$, and $e_{G'_2}(X)+e_{G'_2}(Y)= \min\{k_0, bi(G)\}$. 
Now add some of the edges of $T$ to $G'_1$ and add the remaining ones to $G'_2$ such that among the resulting graphs $G_1$ and $G_2$, the graph $G_i$ is Eulerian. 
Obviosuly, $e_{G_2}(X)+e_{G_2}(Y)\ge \min\{k_0, bi(G)\}$.
Now, assume that $G_1$ is $(m_1+2m_2+2)$-tree-connected.
Again, by applying Theorem~\ref{thm:decomposition:bipartite-index}, the graph $G$ can be decomposed into 
two factors $G'_1$ and $G'_2$ such that $G'_1$ is $m_1$-tree-connected, $G'_2[X,Y]$ is $(m_2+1)$-tree-connected for a bipartition $X,Y$ of $V(G)$, and $e_{G'_2}(X)+e_{G'_2}(Y)= \min\{k_0, bi(G)\}.$
Next, we decompose $G_2[X,Y]$ into two factors $T$ and $F$ such that $T$ is connected and $F$ is $m_2$-tree-connected.
Finally, we add some of the edges of $T$ to $G'_1$ and add the remaining ones to $G'_2$ such that among the resulting graphs $G_1$ and $G_2$, the graph $G_i$ is Eulerian.
}\end{proof}
\begin{lem}{\rm (\cite{ModuloFactorBounded})}\label{lem:low:bi-index:compatible}
{Let $G$ be a $(2k-3)$-edge-connected graph, let $k$ be an integer, $k\ge 2$.
If $f:V(G)\rightarrow Z_k$ is a compatible mapping with respect to a bipartition $X,Y$ of $G$ satisfying $e_G(X)+e_G(Y)< k-1$, then $f$ is compatible with $G$.
}\end{lem}
We also need the following lemma which plays an essential role in this section.
\begin{lem}{\rm (\cite{ModuloFactorBounded})}\label{lem:low:bi-index}
{Let $G$ be a graph, let $k$ be an integer, $k\ge 2$, and let $f:V(G)\rightarrow Z_k$ be a compatible mapping.
If $G[X,Y]$ is $(3k-3)$-edge-connected and $e_G(X)+e_G(Y)\le k-1$ for a bipartition of $X,Y$ of $V(G)$,
then $G$ has an $f$-factor $H$ such that for each vertex~$v$,
$$\lfloor \frac{d_G(v)}{2}\rfloor -(k-1)
\le d_H(v)\le 
 \lceil \frac{d_G(v)}{2}\rceil +(k-1).$$
}\end{lem}
The following result is a weaker version of a result in~\cite{ModuloFactorBounded} for graphs with lower tree-connectivity.
\begin{cor}
{Let $G$ be a graph with $bi(G)\ge k$.
If $G$ is $3k$-tree-connected, then it has  $k$ edge-disjoint non-bipartite spanning Eulerian subgraphs.
}\end{cor}
\begin{proof}
{By Theorem~\ref{thm:decomposition:bipartite-index},  the graph $G$ can be decomposed into two factors  $G_1$, and $G_2$
 such that $G_1$ is $k$-tree-connected, $G_2[X,Y]$ is $k$-tree-connected for a bipartition $X,Y$ of $V(G)$, and $e_{G_2}(X)+e_{G_2}(Y)\ge k$. 
Decompose  $G_1$ into  $k$ connected factors $T_1,\ldots, T_k$, and decompose  $G_2[X,Y]$ into  $k$ connected factors  $T'_1,\ldots, T'_k$.
Take $e_1,\ldots, e_k$ to be $k$ edges of $G_2[X]\cup G_2[Y]$.
Define $H_i=T'_i+e_i$. Note that $H_i$is not bipartite.
Let $F_i$ be a factor of $T_i$ such that for each vertex $v$, $d_{F_i}(v)$ and $d_{H_i}(v)$ have the same parity.
It is enough to set $G_i=F_i\cup H_i$ to construct the desired Eulerian factors.
}\end{proof}
\subsection{Graphs with tree-connectivity at least $2m+2m_0+6k-5$}
The following theorem gives a sufficient edge-connectivity condition for the existence of tree-connected $f$-factors modulo $k$ 
in general graphs with a given bound on degrees.
\begin{thm}\label{thm:higherbi:sharpbound}
{Let $G$ be a graph, let $k$ be an integer, $k\ge 2$, and let $f:V(G)\rightarrow Z_k$ be a compatible mapping. If $G$ is $(2m+2m_0+6k-5)$-tree-connected and $m+m_0>0$, 
then it has an $m$-tree-connected $f$-factor $H$ such that its complement is
 $m_0$-tree-connected and for each vertex~$v$,
$$\lfloor \frac{d_G(v)}{2}\rfloor -(k-1) -m_0\le d_H(v) \le \lceil \frac{d_G(v)}{2}\rceil+(k-1)+m.$$
}\end{thm}
\begin{proof}
{By Corollary~\ref{cor:decomposition:bipartite-index}, the graph $G$ can be decomposed into two graphs $G_1$ and $G_2$ such that $G_1$ is a $(2m+2m_0-1)$-tree-connected Eulerian graph and $G_2[X,Y]$ is $(3k-3)$-tree-connected and $e_{G_2}(X)+e_{G_2}(X) =\min\{k-1, bi(G)\}$ for a bipartition $X,Y$ of $V(G)$.
Obviously, the graph $G_1$ must be $(2m+2m_0)$-edge-connected.
Thus by Theorem~\ref{thm:bounded:partition}, 
the graph $G_1$ has an $m$-tree-connected factor $H_1$
 such that its complement is $m_0$-tree-connected and for each vertex $v$,
$ d_{G_1}(v)/2 -m_0\le d_{H_1}(v) \le d_{G_1}(v)/2+m$.
For each vertex $v$, define $f'(v)=f(v)-d_{H_1}(v)$.
We claim that $f'$ is compatible with $G_2$.
Since $f$ is compatible with $G$, $\sum_{v\in V(G)}f(v)$ and also $\sum_{v\in V(G)}f'(v)$ are even when $k$ is even.
Thus if $bi(G_2)\ge k-1$, then $f'$ is clearly compatible with $G_2$.
Otherwise, $bi(G_2)=bi(G)<k-1$ and so $G_1$ is a bipartite graph with partite sets $X$ and $Y$.
Since the mapping $f$ is compatible with $G$, we may assume that there is an integer number $x$ with $0\le x\le e_G(X)$ such that $\sum_{v\in X}f(v)-2x\stackrel{k}{\equiv}\sum_{v\in Y}f(v)$.
Since $\sum_{v\in X}d_{H_1}(v)\stackrel{k}{\equiv}\sum_{v\in Y}d_{H_1}(v)$, we must have 
$\sum_{v\in X}f'(v)-2x\stackrel{k}{\equiv}\sum_{v\in Y}f'(v)$.
Hence the mapping $f'$ is compatible with $G_2$ with respect to the bipartition $X,Y$ of $V(G)$.
Since $e_{G_2}(X)+e_{G_2}(Y)< k-1$ and $G_2$ is $(2k-3)$-edge-connected, by Lemma~\ref{lem:low:bi-index:compatible}, $f'$ must be compatible with $G_2$ and so the claim holds.
Therefore, by Lemma~\ref{lem:low:bi-index}, the graph $G_2$ has an $f'$-factor $H_2$ such that for each vertex $v$,
$\lfloor d_{G_2}(v)/2\rfloor -(k-1)
\le d_{H_2}(v)\le 
 \lceil d_{G_2}(v)/2\rceil +(k-1).$
Hence $H_1\cup H_2$ is the desired $f$-factor we are looking for.
}\end{proof}
\begin{remark}
{Note that when $k$ is odd, the needed tree-connectivity of Theorem~\ref{thm:higherbi:sharpbound} can be reduced by one by choosing $G_1$ $(2m+2m_0)$-tree-connected (not necessarily Eulerian) and choosing $G_2$ $(3k-4)$-tree-connected and Eulerian.
}\end{remark}
\begin{cor}\label{cor:f,f+k}
{Let $G$ be a $(2m+2m_0+6k-5)$-tree-connected graph, $k\ge m+m_0> 0$, and let $f$ be a positive integer-valued function on $V(G)$. Assume that $f$ is compatibly with $G$ (modulo $k$).
If for each vertex $v$, $f(v)+m_0\le \frac{1}{2}d_G(v)\le f(v)+k-m$,
then $G$ has an $m$-tree-connected factor $H$ such that its complement is $m_0$-tree-connected and for each vertex $v$,
$$d_H(v)\in \{f(v),f(v)+k\}.$$
}\end{cor}
\begin{proof}
{By applying Theorem~\ref{thm:higherbi:sharpbound} when $k\ge 2$ and by applying Theorem~\ref{thm:bounded:partition} when $k=1$, the graph $G$ has an $m$-tree-connected $f$-factor $H$ such that its complement is $m_0$-tree-connected and for each vertex $v$,
 $f(v) -(k-1)\le \lfloor d_G(v)/2\rfloor -(k-1) -m_0\le d_H(v) \le \lceil d_G(v)/2\rceil+(k-1)+m\le f(v)+2k-1$.
Hence $H$ is the desired factor.
}\end{proof}
\begin{lem}\label{lem:regular:bi-index}
{If $G$ is a non-bipartite $r$-regular graph, then
$bi(G)\ge r/2$. Moreover, if $r$ is odd then the lower bound can be replaced by $r$.
}\end{lem}
\begin{proof}
{Let $(X,Y)$ be an arbitrary bipartition of $V(G)$. Since $G$ is non-bipartite, $|X|\neq |Y|$.
We may assume that $|X|>|Y|$. 
Since $G$ is $r$-regular, $r|X|-2e_G(X)=r|Y|-2e_G(Y)$, which implies that $e_G(X)\ge r(|X|-|Y|)/2\ge r/2$.
If $r$ is odd, then $|X|-|Y|$ must be even and so $|X|\ge |Y|+2$ which implies that $e_G(X)\ge r(|X|-|Y|)/2\ge r$.
This completes the proof.
}\end{proof}
Now, we are ready to conclude the following interesting result for regular graphs.
\begin{cor}
{Let $G$ be an $r$-regular graph and let $a$ and $b$ be two positive integers satisfying $a+m_0 \le r/2\le b-m$ and $m+m_0> 0$. If $G$ is $(2m+2m_0+6k-5)$-tree-connected and $ab|V(G)|$ is even,
then $G$ has an $m$-tree-connected $\{a,b\}$-factor such that its complement is $m_0$-tree-connected.
}\end{cor}
\begin{proof}
{Let $k=b-a$ so that $k\ge 1$. Define $f=a$ to be a constant mapping on $V(G)$. 
Since $bi(G)\ge r/2\ge 3k-3\ge k-1$, by Lemma~\ref{lem:regular:bi-index}, the mapping $f$ must be compatible with $G$ (mod $k$).
Thus by Corollary~\ref{cor:f,f+k}, the graph $G$ has an $m$-tree-connected $\{f,f+k\}$-factor $H$ such that its complement is $m_0$-tree-connected. Hence $H$ is the desired $\{a,b\}$-factor.
}\end{proof}
\section{Bipartite tree-connected modulo regular factors}
\label{sec:modulo-regular}
In this section, we consider the existence of bipartite connected  modulo regular factors in tree-connected graphs. 
We begin with the following well-known theorem about the existence of connected even factors.
\begin{thm}\rm{(Jaeger~\cite{Jaeger-1979})}\label{thm:Eulerian}
{Every $2$-tree-connected graph admits a spanning Eulerian subgraph.
}\end{thm} 
For graphs with a bit higher tree-connectivity, one can find bipartite connected modulo $2$-regular factors as the following corollary. Note that this result also holds for $7$-edge-connected graphs, by replacing Proposition 1 in \cite{Thomassen-2008-P4} in the proof. On the other hand, there are $4$-edge-connected graphs with no bipartite spanning Eulerian subgraph; for example, the complete graph of order five. It remains to decide whether this result holds for $3$-tree-connected or $5$-edge-connected graphs.
\begin{cor}
{Every $4$-tree-connected graph admits a bipartite spanning Eulerian subgraph.
}\end{cor} 
\begin{proof}
{By Lemma~\ref{lem:bipartite:factor}, the graph $G$ contains a $2$-tree-connected bipartite factor $H$.
Thus by Theorem~\ref{thm:Eulerian}, the bipartite graph $H$ admits a bipartite spanning Eulerian subgraph and so does $G$. 
}\end{proof}
The following theorem gives sufficient edge-connectivity conditions for the existence of a bipartite tree-connected $f$-factor. 
\begin{thm}\label{thm:m+4k-4}
{Let $G$ be a graph with $z\in V(G)$, let $k$ be an integer, $k\ge 2$, and let $f:V(G)\rightarrow Z_k$ be a compatible mapping. 
If $G$ is $(m+m_0+4k-4)$-tree-connected, then it has an $m$-tree-connected $f$-factor $H$ such that its complement is $m_0$-tree-connected and for each $v\in V(G)\setminus \{z\}$, $d_H(v)\le d_G(v)-(k-1)$.
If $G$ is bipartite, the needed tree-connectivity can be reduced to $m+m_0+2k-2$.
}\end{thm}
\begin{proof}
{By applying Theorem~\ref{thm:decomposition:bipartite-index}, the graph $G$ can be decomposed into three factors $G_0$, $G_1$, and $G_2$ such that $G_0$ is $m_0$-tree-connected, $G_1$ is $m$-tree-connected, and $G_2[X,Y]$ is $(2k-2)$-tree-connected and $e_{G_2}(X)+e_{G_2}(Y)=\min\{k-1, bi(G)\}$ for a bipartition $X,Y$ of $V(G)$.
For each vertex $v$, define $f'(v)=f(v)-d_{G_1}(v)$.
Since $f$ is compatible mapping by $G$, one can show that $f'$ is also compatible with $G_2$ using the arguments stated in the proof of Theorem~\ref{thm:higherbi:sharpbound} for proving a claim. 
Thus by Theorem~\ref{thm:bi:k:partition}, the graph $G_2$ has an $f'$-factor $F$ such that  for each $v\in V(G)\setminus \{z\}$, $d_F(v)\le d_{G_2}(v)-(k-1)$.
Hence the graph $G_1\cup F$ is the desired factor.
}\end{proof}
The following corollary gives sufficient tree-connectivity conditions for the existence of a connected modulo $k$-regular factors. 
\begin{cor}
{Every $(4k-3)$-tree-connected graph admits a connected modulo $k$-regular factor, and every $(4k-2)$-tree-connected graph admits a bipartite connected  modulo $k$-regular factor.
}\end{cor}
\begin{proof}
{Apply Theorem~\ref{thm:m+4k-4} with $m=1$, $m_0=0$, and $f=0$.
}\end{proof}
%
%
%
%
%
%
%
%
%
%
%
%
%
%
%
%

\end{document}